\documentclass{article}%
\usepackage{amsmath}
\usepackage{amsfonts}
\usepackage{amssymb}
\usepackage{amsthm}
\usepackage{mathtools}
\usepackage{graphicx}%
\providecommand{\U}[1]{\protect\rule{.1in}{.1in}}
\providecommand{\U}[1]{\protect\rule{.1in}{.1in}}

\newtheorem{theorem}{Theorem}

\newtheorem{definition}[theorem]{Definition}

\newtheorem{lemma}[theorem]{Lemma}

\newtheorem{proposition}[theorem]{Proposition}
\newtheorem{remark}[theorem]{Remark}

\everymath{\displaystyle}

\begin{document}

\title{Complex Germen on invariant isotropic tori under the Hamiltonian phases flow  with in involution Hamilton functions }

\author{\ A. C. Alvarez
	\thanks{Instituto Nacional de Matem\'{a}tica Pura e Aplicada, Estrada Dona
		Castorina 110, 22460-320 Rio de Janeiro, RJ, Brazil. E-mail:
		amaury@impa.br} , \  Baldomero Vali\~no Alonso \thanks{Havana University}
}

\maketitle

\begin{abstract}
	M. M. Nekhoroshev put forward the problem of to find the Complex Germ on a isotropic invariant torus with respect to Hamiltonian phases flows which come from k-functions in involution. This statement was partially solved in \cite{pno} establishing that if certain simplectic operator has 
a simple spectrum then the complex germ exist. In this work we solve this problem, providing a full solution, i.e. we present conditions for the existence and uniqueness of complex germ through the monodromy  operator constructed in \cite{pno}, but without the simple spectrum condition. We study also the Hamiltonian system with cyclic variables.
\end{abstract}

\section{Introduction}
In several problems of the quantum and theoretical physics approximated solution to the partial
differential equations which contain a small parameter in the higher derivative order are obtained, as well as to approximated eigenvalues
and eigenvector of self-adjoint differential operator which depend on a small parameter. In such problems have been used
the asymptotic methods \cite{maslov,maslov2,maslov3}, which nowadays are developed widely in several branches of the physics-mathematics.
It is well known the success of asymptotic methods, e.g. with the quantification method was solved
the older and sharp problem of the mechanic classic: calculation of the energetic level of the hydrogen atom \cite{belov}. 

In \cite{maslov2} over a 2n-dimensional phases space to obtain an asymptotic quasiclassical solution with respect
to a small parameter on an isotropic tori k-dimensional ($k < n$) is obtained. This asymptotic on a torus is accomplished with a new geometric
object which was called the Complex Germen .i.e. a family of complex planes with certain properties. Such object does
not exist over any isotropic manifold. In such sense, V. P. Maslov put forward and solved the problem about its existence and
construction techniques.
At the same time, the uniqueness of Complex Germ has a great signification such that asymptotic be well defined. In \cite{bal1} were completed
the results obtained in \cite{maslov2}, as well as have solved the uniqueness problem for the singular point and a closed
trajectory of the Hamiltonian system. Further, was solved of existence and uniqueness of Germen on
isotropic torus to Hamiltonian with cyclic variable.

Some more later, M. M. Nekhoroshev put forward this problem but on isotropic invariant torus with respect to Hamiltonian phases flows
which come from k-functions in involution. This statement was partially solved in \cite{pno} establishing  that if certain
simplectic operator has a simple spectrum then the complex germ exist. In this work we solve this problem,
providing a full solution, i.e. we present conditions for the existence and uniqueness of complex germ through the monodromy operator constructed
in \cite{pno}, but without the simple spectrum condition. We study also the Hamiltonian system with cyclic variables.

\section{Preliminaries} % Main chapter title
\label{Chapter2} % For referencing the chapter elsewhere, use \ref{Chapter1} 

%----------------------------------------------------------------------------------------

% Define some commands to keep the formatting separated from the content 
%----------------------------------------------------------------------------------------

Let $M^{2n}$ be a $2n$-dimensional differentiable manifold with local coordinates
$(p,q)$. We assume $C^{\infty}$ manifolds and scalar functions and vectorial fields as well. Let us introduce basic results, details can be found in \cite{arnold1}. Let $T_mM^{2n}$ be the tangent space at point $m$ of manifold $M^{2n}$. We introduce natural manifold on $TM^{2n}=\displaystyle \cup_{m \in M^{2n}}T_mM^{2n}$.
\begin{definition}
 It is called exterior form of degree two or 2-form at point $m$ on the manifold $M^{2n}$ to
 the bilinear and
 antisymmetric application $\omega^2: T_mM^{2n} \times T_mM^{2n} \rightarrow \Re$ , i.e.
 
 \begin{itemize}
  \item $\omega^2(\alpha x + \beta y,z)=\alpha \omega^2(x,z)+\beta\omega^2(y,z)$,
  \item $\omega^2(x,y)=-\omega^2(y,x)$,
   \end{itemize}
for all $x,y,z \in T_mM^{2n}$ and $\alpha, \beta \in \Re$ (see \cite{arnold2}).
\end{definition}

A 2-form $\omega^2$ is closed if $d\omega^2\equiv0$, where $d : \Omega^2(M^{2n}) \rightarrow \Omega^3(M^{2n})$ is
an operator of exterior differentiation on the space of the 2-form $\Omega^2(M^{2n})$. Besides, it is called
non-degenerate if  $\omega^2(x,y)=0$ for all $x \in T_mM^{2n}$ then $y=0$.

\begin{definition}
 A closed, non-degenerate and differential 2-form $\omega^2$ on the manifold $M^{2n}$ is called symplectic structure. The couple $(M^{2n},\omega^2)$ is called symplectic manifold and the tangent space 
 in each point $m$ of the manifold is called symplectic vectorial field whose symplectic structure is the restriction
 of $\omega^2$ to $T_mM^{2n} \times T_mM^{2n}$.
\end{definition}
One example of symplectic structure consist of $M^{2n}=\Re^{2n}$ with $\omega^2=dq \wedge dp$, where
\begin{equation}
 dq \wedge dp(x,y)= \displaystyle \sum_{i=1}^{n}dq_i(x)dp_i(y)-dq_i(y)dp_i(x), \quad \forall x,y \in T_m\Re^{2n}
\end{equation}
and  for all $ m \in \Re^{2n}$. In this case $T_m\Re^{2n}$ is identified with $\Re^{2n}$, where the 1-forms $dq_i$ and $dp_i$
are defined as
\begin{equation}
 dq_i(x)=q_i(x) \quad \text{and} \quad  dp_i(x)=p_i(x), 
\end{equation}
where $q_i$ and $p_i$: $\Re^{2n} \rightarrow \Re$ are the coordinates system. i.e. $q_i(x)$ ($p_i(x)$) are $i$-th ($n+i$-th) coordinates
of the vector $x$ in a prefixed basis of the real linear space $\Re^{2n}$. The symplectic structure defined in this way is called \textit{standard}.

\begin{definition}
 The coordinates of the local chart $(q_1,\ldots,q_n,p_1,\ldots,p_n)$  of a symplectic manifold are called canonical if
 the expression of the symplectic structure $\omega^2$ in this coordinate system coincides with the standard.
 \label{poi2a}
\end{definition}
As consequence of the Darboux Theorem, in each point of a symplectic manifold, there is a neighbour with canonical coordinates \cite{arnold1}.

\begin{definition}
 Let $M^{2n}$ be a symplectic manifold and let $T_mM^{2n}$ and $T_m^{*}M^{2n}$ be tangent and cotangent spaces at point $m \in M^{2n}$,
 we define the operator $J : T_m^{*}M^{2n} \rightarrow T_mM^{2n}$ as
 \begin{equation}
  \omega^2(x,Ja)=a(x), \quad \text{for} \quad x \in  T_mM^{2n} ; a \in  T_m^{*}M^{2n}.
 \end{equation}
 The operator $J$ defined in this way constitutes an isomorphism between vectorial spaces.
  \label{poi2}
\end{definition}

\begin{definition}
 The Poisson bracket of two functions $F$, $G$ over the manifold $M^{2n}$ is defined as
 
\begin{equation}
 [F,G]=\omega^2(JdF,JdG),
\end{equation}
where $dF$ and $dG$ are the differentials 1-forms of $F$ and $G$ on $M^{2n}$.
\end{definition}
If $[F,G]=0$, we say that functions $F$ and $G$ are in involution.

From Definition  \ref{poi2}, we have
\begin{equation}
 [F,G]=-dG(JdF)=dF(JdG).
\end{equation}
\begin{definition}
 A Hamiltonian system is the triple $(M^{2n},\omega^2,H)$, where $(M^{2n},\omega^2)$ is a symplectic manifold and functions $H$
 is defined on it. The field $JdH$ is called the Hamiltonian vectorial field. 
\end{definition}

The matrix of the Hamiltonian operator $H$ in canonical coordinates is

\[ \left( \begin{array}{cc}
0 & -I_n  \\
I_n & 0 
\end{array} \right)\] 
where $0$ and $I_n$ denotes the zero and identity n-dimensional matrices. Thus, in canonical coordinates
the Hamiltonian system takes the form 
\begin{equation}
\dot{p}=H_q, ~~ \dot{q}=-H_p,
\end{equation}
where $H_q=(\frac{\partial H}{\partial q_1},\ldots,\frac{\partial H}{\partial q_n})$
and $H_p=(\frac{\partial H}{\partial p_1},\ldots,\frac{\partial H}{\partial p_n})$.

Let us assume that the solution of the Hamiltonian system $(M^{2n},\omega^2,H)$ can be extended
to $\Re$, i.e. for $-\infty < t < +\infty$. In this case, $g^t_Hm$ denotes the value of the solution $\theta$ with initial
condition $\theta(0)=m$, we obtain the application $g^t_H M^{2n} \rightarrow M^{2n}$ for $t$ fixed. This application constitute a one parametric group of
diffeomorphism, i.e.
$g^0_Hm=m$ and $g^{t+s}_Hm=g^t_Hm \circ g^s_Hm$, for all $m \in M^{2n}$. This group is called flow of phases of the Hamiltonian system.
Moreover, the application $g^t_M$ is symplectic for each $t$ fixed, i.e. $(g^t_H)^*\omega^2=\omega^2$, where
\begin{equation}
 (g^t_H)^*\omega^2(x,y)=\omega^2((g^t_H)_{*,m}x,(g^t_H)_{*,m}y),
\end{equation}
for all $m \in M^{2n}$. Here $(g^t_H)_{*,m}$ denotes the derivatives of $(g^t_H)$ for each $t \in \Re$  (see \cite{arnold2}).
\begin{definition}
	Let $\varLambda$ a submanifold of $M^{2n}$ and $T_m\varLambda$ the tangent space of $\varLambda$ at point $m$. The submanifold $\varLambda$ is called isotropic if the symplectic structure $\omega^2$ is null on it, i.e.,
 $\omega^2(x,y)=0, \forall x,y \in T_m\varLambda$ and $m \in M^{2n}$.
\end{definition}

The submanifold is invariant respect to the Hamiltonian system $(M^{2n},\omega^2,H)$ if $IdH(m) \in T_m\varLambda, \forall m \in \varLambda$,
which in terms of the phases flow is rewritten as $g^t_H(\varLambda)=\varLambda$, for all $t \in \Re$, hence
$(g^t_H)_{*,m}(T_m\varLambda)=T_m\varLambda$.

Let us consider the complexification of the linear space $\Re^2$ and a linear operator for a positive integer $n$. The complexification of $\Re^n$ is a n-dimensional linear space $(\Re^{n})^\mathbb{C}$ constructed as follow:
the point in $(\Re^{n})^\mathbb{C}$ is denotes either by $(x,y)$ or $x+iy$ , where $x,y \in \Re^n$.

If $\alpha \cdot V$ denotes the multiplication of a scalar by $V \in \Re^n$ and $V+W$ the sum of two any vectors in $\Re^n$,
the the multiplication of complex scalar $\alpha+i \beta$ by a vector $V+iW \in (\Re^{n})^\mathbb{C}$ and the sum of
two vectors $V_1+iW_1$, $V_2+iW_2$ are defined as:
\begin{align}
 (\alpha+i \beta) \cdot (V+iW)= (\alpha V - \beta W)+ i (\alpha W + \beta V),\\
 (V_1+iW_1) \cdot (V_2+iW_2)= (V_1 V_2 - W_1 W_2)+ i (V_1 W_2 + W_2 V_2),
\end{align}
thus $(\Re^{n})^\mathbb{C}$ constitutes an complex lineal space $((\Re^{n})^\mathbb{C}=\mathbb{C}^n)$.

\begin{remark}
Similarly, for any real vectorial subspace $P$ is possible to define its complexification $P^\mathbb{C}$.
\end{remark}

It is well known, that between the tangent space $T_m(M^{2n})$ to submanifold 2n-dimensional $M^{2n}$ and $\Re^{2n}$ there is an isomorphism  $X : T_m(M^{2n}) \rightarrow \Re^{2n}$. Analogously, we can establish an isomorphism between the complexifications $T_m(M^{2n})^ \mathbb{C}$ and $\Re^\mathbb{C}$.

Let us denote by $A: \Re^m \rightarrow \Re^d$ a $\Re$-linear operator. A complexification of the operator $A$ is a $\mathbb{C}$-linear operator  $A^{\mathbb{C}}: (\Re^m)^\mathbb{C} \rightarrow (\Re^d)^\mathbb{C}$ defined by the relation $A^\mathbb{C}(x+iy)=Ax+iAy$. The following relations are valid:
$(A+B)^\mathbb{C}=A^\mathbb{C}+B^\mathbb{C}$, where $A$ and $B$ are $\Re$-linear operators.

Let us introduce the following concepts: let $\omega^2$ be a real simplectic structure on the manifold 
$M^{2n}$. We call complexification of $\omega^2$ defined on $T_m(M^{2n})^\mathbb{C}$, for all $m \in M^{2n}$ to the form $\omega^\mathbb{C} : T_m(M^{2n})^\mathbb{C} \times T_m(M^{2n})^ \mathbb{C} \rightarrow \mathbb{C}$, given by
\begin{equation}
(\omega^2)^\mathbb{C}(u+iv,x+iy)=(\omega^2(u,x)-\omega^2(v,y))+i(\omega^2(u,y)+\omega^2(v,x)),
\end{equation}
$\forall u+iv,x+iy \in  T_m(M^{2n})^\mathbb{C}$.

\begin{remark}
Since $\omega^2$ is antisimetric then $(1/2i)(\omega^2)^\mathbb{C}(x,\bar{x})$ is real for all $x \neq 0$ in $T_m(M^{2n})$. We denote $(\omega^2)^\mathbb{C}(x,y)=[x,y]$. It is possible to verify that the complexification of the standard simplectic structure on $T_m(M^{2n})$ is the standard simplectic structure on $T_m(M^{2n})^\mathbb{C}$ at any point $m$ of the manifold $M^{2n}$ (\cite{gon}).
\end{remark}

\section{Statement of the problem}

Let $M^{2n}$ be a $2n$-dimensional simplectic manifold and $F_1,\ldots,F_k$ a family of functions defined on $M^{2n}$: $F_j: M^{2n} \rightarrow \Re$, $j=1,\ldots,k$ such that $k<n$, which stay in involution on $M^{2n}$. Let the $k$-Hamiltonian systems $(M^{2n},\omega^2,F_j)$ with the correspond Hamiltonian flux of phases
${g^t_{F_j}}, t \in \Re; j=1,\ldots,k$. 
Since the functions $F_j, j=1,\ldots,k$ are in involution, the flux of phases commute, i.e.
$g^t_{F_i} o g^t_{F_j} = g^t_{F_j} o g^t_{F_i}$,  for all $t \in (-\infty, + \infty)$
and for all $i,j=1,\ldots,k$.

We assume that

\begin{itemize}
\item The  Hamiltonian flux phases  ${g^t_{F_j}}$, $j=1,\ldots,k$ are global, i.e. they are defined for all $t \in (-\infty, + \infty)$. It is valid for example if we assume that
the manifold $M^{2n}$ is compact.
\item The torus $\Lambda^k=S^1 \times, \ldots, \times S^1$ (where $S^1$ denotes the unit circle),
$k <n$ is a $k$-dimensional isotropic submanifold of $M^{2n}$ ($\Lambda^k \subset M^{2n}$) and is
invariant respect to the Hamiltonian system $(M^{2n},\omega^2,F_j), j=1,\ldots,k$.

\item The differential $dF_1,\ldots,dF_k$ are lineal independent at each point of the manifold $M^{2n}$.

\end{itemize}

For the construction of asymptotic solution of several partial differential equation in \cite{maslov2} the concept of Complex Germ on isotropic manifold is introduced. The issues
of existence and uniqueness of such object is treated in this work. The main difficulty is that
not always exist the Germ over any isotropic manifold.
\begin{definition}
A Complex Germ over the isotropic tori $\Lambda^k (k<n)$ is a smooth map on $\Lambda^k$,
$r^n: m \rightarrow r^n(m), \forall m \in \Lambda^k$, such that to each point $m \in \Lambda^k$ correspond a $n$-dimensional complex subspace $r^r(m)$ of the complexification 
of the tangent space to $M^{2n}$ at the point $m$ ($T_m(M^{2n})^{\mathbb{C}}$) with following properties:

i) $r^n(m)$ is a lagrangian subspace, i.e. $dim(r^n(m))=n$ and isotropic ($[x,y]=0, \forall x,y \in r^n(m)$),

ii) $r^n(m) \supset T_m(\Lambda^k)^\mathbb{C}$,

iii) $r^n(m)$ is disipative respect of $T_m(\Lambda^k)$, i.e.
\begin{equation}
\forall x \in r^n(m) \setminus T_m(\Lambda^k)^\mathbb{C} \quad \text{holds} \quad  (1/2i)[x,\bar{x}]>0,
\label{der2a}
\end{equation} 

iv) $r^n(m)$ is invariant respect the Hamiltonian flux ${g^t_H}, t \in \Re$ of a given function $H$,i.e. $\forall \in \Lambda^k, \forall t \in (-\infty,+\infty)$ holds
\begin{equation}
[(g^t_H)_{*,m}]^\mathbb{C}(r^n(m))=r(g^t_Hm),
\label{der1a}
\end{equation}  
where  $[(g^t_H)_{*,m}]^\mathbb{C}$ is the complexification of the derivative at point $m$
of an element $g^t_H$ of Hamiltonian flux phases $H$ associated to Hamilton function $H$. 
\label{def1}
 \end{definition}

\begin{remark}
From now we omit the supraindex $\mathbb{C}$ that indicated the complexification, in the space and operator. But implicitly we use the properties enunciated previously.
\end{remark}

The aim of this paper is to seek conditions for the existence and uniqueness of a Complex Germ on the a invariant torus to Hamiltonian system $(M^{2n},\omega^2,F_j),~j=1,\ldots,k; (k<n)$.

\section{Condition for the existence of the Complex Germ on the torus}
Let 
\begin{equation}
\Sigma=\{m \in M^{2n} : F_i(m)=f_i, i=1,\ldots,k, f=(f_1,\ldots,f_k) \in \Re^k  \},
\end{equation}
be the intersection of the level surfaces defined by the functions $F_j, j=1,\ldots,k$ which contain the trajectories of the Hamiltonian system  $\displaystyle x^{•}=JdF_j(x), j=1,\ldots,k$. Since $dF_1,\ldots,dF_k$ are linearly independent then $\Sigma$ is a submanifold submerge  in
$M^{2n}$ of dimension $2n-k$.

Let us define the action of additive group $\Re^k$ over $M^{2n}$, as follow:
to each $\overrightarrow{t}=(t_1,\ldots,t_k) \in \Re^k$ correspond the difemeorphism of $M^{2n}$ as:
\begin{equation}
g^{\overrightarrow{t}}=g_1^{t_1}o \ldots o g_k^{t_k},
\end{equation}
where $g_j^t=g_{F_j}^t, \forall j=1,\dots,k$. Since $g_j^t$ are simplectic diffemeorphism 
the $g^{\overrightarrow{t}}$ as well, i.e. $(g^{\overrightarrow{t}})^*\omega^2=\omega^2, \forall \overrightarrow{t} \in \Re^k $.
\subsection{Monodromy operator}
The condition for existence of the Complex Germ are given in term of the monodromy operator which we defined in the following paragraph

Let us fix the point $m \in \Lambda^k$. Denoting by $G$ the discrete subgroup of $\Re^k$ defined as follow
\begin{equation}
G=\{\overrightarrow{t} \in \Re^k : g^{\overrightarrow{t}}m=m\},
\label{mono1a}
\end{equation} 
which does not depend on the choice of the point $m$. The subgroup $G$ can be generated by a set
of $k$ elements linearly independent(see \cite{arnold1}), i.e. 
\begin{equation}
G=\{l_1T_1+\ldots,l_kT_k;~ l_i~ \in \Re;~~ T_i \in G~~ \text{are linearly independent},~~ i=1,\ldots,k\}.
\end{equation} 
Since $\Lambda^k$ is invariant respect to Hamiltonian system $x^{\prime}=JdF_j(x);j=1,\ldots,k$ we have $g_j^t\Lambda^k=\Lambda^k$ therefore $g^{\overrightarrow{t}}\Lambda^k=\Lambda^k$.
We have also that the map $f:\Re^k \rightarrow M^{2n}$ given by $f(t)=g^{\overrightarrow{t}}m$
which to any $\overrightarrow{t}$ correspond a point in $\Lambda^k$ is sobrejective, but is not injective because $\Lambda^k$ is compact and $\Re^k$ is not; therefore there exist $t_1,t_2 \in \Re^k; t_1 \neq t_2$, such that $g^{t_1}m=g^{t_2}m$ and $g^{t_1-t_2}m=m$ with $t_1-t_2 \in G$;
which mean that the subgroup $G$ is not trivial.
\begin{definition}
Let $T \in G$. The operator $G_m=(g^T_{*,m}):T_m(M^{2n}) \rightarrow T_m(M^{2n})$ is called
monodromy operator of the subgroup $G$ at point $m \in \Lambda^k$.
\label{mono}
\end{definition}
Let us rewrite the monodromy operator in a way more amassing to describe the sufficient condition for the existence of the Complex Germ.
We gave
\begin{equation}
g^T_{*,m}=g_{*,m}^{l_1T_1+\ldots+l_kT_k},
\end{equation}
for certain $l=(l_1,\ldots,l_k) \in Z^k$, where $T_1,\ldots,T_k$ are the generator of
subgroup $G$. Let denote $T_j=(t_{j1},\ldots,t_{jk}); t_{ij} \Re; \forall i,j=1,\ldots,k$.
thus we obtain
\begin{equation}
g_{*,m}^{l_1T_1+\ldots+l_kT_k}=g_{1*,m}^{l_1t_{11}+\ldots+l_kt_{k1}}o \ldots og_{k*,m}^{l_1t_{1k}+\ldots+l_kT_{kk}}.
\label{ant1}
\end{equation}
Using that $g_j^{t+s}=g_j^tog_j^s$ and $g_j^tog_i^t=g_i^tog_j^t$ and reordering \eqref{ant1} we obtain 
\begin{equation}
g^T_{*,m}=(g_{1*,m}^{l_1t_{11}}+\ldots+g_{k*,m}^{l_1t_{1k}})o(g_{1*,m}^{l_2t_{21}}+\ldots+g_{k*,m}^{l_2t_{2k}})o\ldots
o(g_{1*,m}^{l_kt_{k1}}+\ldots+g_{1*,m}^{l_kt_{kk}}),
\end{equation} 
or
\begin{equation}
G_m=G_1^{l_1}o \ldots o G_k^{l_k},
\label{princ}
\end{equation}
where $G_j=g_{*,m}^{T_j}: T_m(M^{2n}) \rightarrow T_m(M^{2n}), j=1,\ldots,k$ is called monodromy operator with period $T_j$.
The following is valid 
\begin{lemma}
Over the torus $\Lambda_k$ is possible to defined $k$ vectorial fields linearly independent.
\end{lemma}
\begin{proof}
Since $\Lambda^k$ is invariant respect to the $k$- Hamiltonian system
$\dot{x}=JdfF_{j}(x)$,\\$j=1,\ldots,k$, $JdF_{j}(m) \in T_m(\Lambda^k), \forall m \in \Lambda^k$. Also, since $dF_1,\ldots,dF_k$ are linearly independent at each 
point $m \in \Lambda^k$ and the operator $J$ is regular we obtain
that the vector $JdF_j(m),j=1,\ldots,k$ are linearly in each point
of $m \in \Lambda^k$.
\end{proof}

Note that $JdF_j,j=1,\dots,k$ in each point $m \in \Lambda^k$ constitute a basis of the tangent space $T_m(\Lambda^k)$ of $\Lambda^k$. Using the $\Lambda^k$ is invariant we obtain
\begin{equation}
G_j(JdF_i)=\displaystyle \sum_{n=1}^{k}\beta_nJdF_n,
\label{erq1}
\end{equation}
 \begin{equation}
G_m(JdF_i)=\displaystyle \sum_{n=1}^{k}\mu_nJdF_n,
\label{erq2}
\end{equation}
with $\beta_n,\mu_n \in \mathbb{C};n=1,\ldots,k$.
Let $\Gamma_m=T_m(\Sigma) \diagup T_m(\Lambda^k)$; we have $dim(\Gamma_m)=2(n-k)$.

It is valid the following
\begin{lemma}
Let $[\theta] \in \Gamma_m$ then $G_j([\theta])=[G_j(\theta)]$ and $G_m([\theta])=[G_m(\theta)]$. 
\label{lem13}
\end{lemma}
\begin{proof}
Let $[\theta]=\{ \theta : \theta^{'} \equiv \theta mod(T_m(\Lambda^k))\}$,
$\theta^{'} \equiv \theta mod(T_m(\Lambda^k))$ if and only if $\theta^{'}=\theta+\displaystyle \sum_{i=1}^{k}\beta_i JdF_i$, where $\beta_i \in \mathbb{C},i=1,\ldots,k$,
Also the following equality is valid 
\begin{equation}
G_j(\theta^{'})=G_j(\theta)+ \displaystyle \sum_{i=1}^{k}\mu_iJdF_i,
\end{equation}
and by using equality \eqref{erq1} we have $G_j(\theta^{'})=G_j(\theta)+\sum_{i=1}^{k} \mu_i JdF_i$ for certain $\mu_i \in \mathbb{C}, i=1,\ldots,k$ and due to \eqref{princ} also satisfy
that
\begin{equation}
G_m(\theta^{'})=G_m(\theta)+ \displaystyle \sum_{i=1}^{k} \tau_i JdF_i,
\end{equation}
\begin{equation}
\Xi_m: \Gamma_m \rightarrow \Gamma_m,  \quad \text{such that} \quad [\theta] \rightarrow [G_m(\theta)],
\end{equation}
\begin{equation}
\Xi_j: \Gamma_j \rightarrow \Gamma_j,  \quad \text{such that} \quad [\theta] \rightarrow [G_j(\theta)],
\end{equation}
with $j=1,\ldots,k$.
From now, the structure $[,]$ must be defined between equivalence class modulo $T_m(\Lambda)$
on each point $m\ \in \Lambda^k$. To do so, we need to prove that $[,]$ is compatible with respect to equivalence class i.e. if $\theta \equiv \sigma mod(T_m(\Lambda^k))$ and
$\theta^{'} \equiv \sigma^{'} mod(T_m(\Lambda^k))$, then $[\theta,\theta^{'}]=[\sigma,\sigma^{'}]$. We assume by definition that $[[\theta],[\theta]]=[\theta,\theta]$.
We prove that  $[\theta,\theta^{'}]=[\sigma,\sigma^{'}]$:
\begin{equation}
\theta=\sigma+ \displaystyle \sum_{i=1}^{k}\alpha_i JdF_i, \quad \theta^{'}=\sigma^{'}+ \displaystyle \sum_{i=1}^{k}\delta_i JdF_i.
\end{equation}
By using bi-linearity property we have
\begin{equation*}
[\theta,\theta^{'}]=[\displaystyle \sum_{i=1}^{k}\alpha_i JdF_i,\displaystyle \sum_{i=1}^{k}\delta_i JdF_i],
\end{equation*}
or
\begin{eqnarray}
[\theta,\theta^{'}]=[\sigma,\sigma^{'}]+[\displaystyle \sum_{i=1}^{k}\alpha_i JdF_i,\displaystyle \sum_{i=1}^{k}\delta_i JdF_i],\\
+[\sigma,\displaystyle \sum_{i=1}^{k}\delta_i JdF_i]+[\displaystyle \sum_{i=1}^{k}\alpha_i JdF_i,\sigma^{'}].
\end{eqnarray}
We have
\begin{equation}
[\sigma,\displaystyle \sum_{i=1}^{k}\delta_i JdF_i]=\displaystyle \sum_{i=1}^{k}\delta_i [\sigma,JdF_i]=\displaystyle \sum_{i=1}^{k}\delta_i JdF_i(\sigma)=0,
\end{equation}
because $JdF_i(\sigma)=0$ since $\sigma \in T_m(\Sigma)$. Analogously is prove that $[\displaystyle \sum_{i=1}^{k}\alpha_i JdF_i,\sigma^{'}]=0$.  Also that $[\displaystyle \sum_{i=1}^{k}\alpha_i JdF_i,\displaystyle \sum_{i=1}^{k}\delta_i JdF_i]=0$   holds , since the functions $F_i, i=1,\ldots,k$ stay in involution. Finally we have the proof.

\end{proof}

\begin{definition}
The operator $\Xi_m: \Gamma_m \rightarrow \Gamma_m$ is called reduced monodromy operator at point $m$. i.e.,
\begin{equation}
\Xi_m=(\Xi_1)^{l_1} o \ldots o (\Xi_k)^{l_k},
\label{redu1}
\end{equation}
where the operators $\Xi_j$, $j=1,\ldots,k$ are called
 reduced monodromy operator with period $T_j$. 
 \label{redu}
\end{definition}
Also the following proposition is valid 
\begin{proposition}
	The quotient space $\Gamma_m=T_m(\Sigma)/T_m(\Lambda^k)$ has a natural symplectic structure such that
	the operators  $\Xi_j$, $j=1,\ldots,k$ are symplectic respect to the structure $[,]$ of the space
	$T_m(M^{2n})$ induce over this space.
\end{proposition}

\begin{proof}
	It is known that if the vectorial space $V$ is given a bilinear, antisymmetric, degenerate form then over the quotient
	vectorial space $V/V^\perp$ is induced a bilinear, antisymmetric, no degenerate form. Taking $V=T_m(\Sigma)$
	and the bilinear form $[,]$ which is degenerate over $T_m(\Sigma)$ and we use that $T_m(\Sigma)^\perp=T_m(\Lambda^k)$.
Let $\sigma=\sum_{i=1}^{k}\sigma_i JdF_i$, which $\sigma_i \neq 0$ for some $i=1,\ldots,k$. Since $\sigma \in T_m(\Lambda^k)$ holds that $[x,\sigma]=0$,~ $\forall x \in T_m(\Sigma)$ because $[x,IJdF_i]=dF_i(x)=0$ ($F_i$ is constant on $\Sigma$). This proved that $[,]$ is degenerate on $T_m(\Sigma)$ and $JdF_i \in T_m(\Sigma)$. Since
 $\dim(T_m(\Sigma))=2n-k$ then  $\dim(T_m(\Sigma)^\bot)=k$ and $IJdF_i$ ($i=1,\ldots,k$) constitutes a basis of 
 $T_m(\Lambda^k)$. Since $T_m(\Sigma)^\bot=T_m(\Lambda^k)$ we have that $\Gamma_m$ is a subspace lineal symplectic.
	Now, we verify that the operator $\Xi_j$ ($j=1,\ldots,k$) are symplectic. Let $[\theta], [\theta^{'}] \in \Gamma_m$.
	we have 
	\begin{equation*}
	[\Xi_j([\theta]),\Xi_j([\theta^{'}])]=	[G_j([\theta]),G_j([\theta^{'}])],
	\end{equation*}
where $G_j$, with $j=1,\ldots$ are the monodromy operators in Definition \ref{mono}. From Lemma \ref{lem13} we have
	\begin{equation*}
	[\Xi_j([\theta]),\Xi_j([\theta^{'}])]=	[[G_j(\theta)],[G_j([\theta^{'})]],
	\end{equation*}
By definition of the product of two class we obtain
	\begin{equation*}
	[\Xi_j([\theta]),\Xi_j([\theta^{'}])]=	[G_j(\theta),G_j(\theta^{'})].
	\end{equation*}
Since $G_j$ are symplectic we have $[\Xi_j([\theta]),\Xi_j([\theta^{'}])]=[\theta,\theta^{'}]$. Using again
the definition of the product of two class we have $[\Xi_j([\theta]),\Xi_j([\theta^{'}])]=[[\theta],[\theta^{'}]]$; so
$\Xi_j$, with $j=1,\ldots,k$ are symplectic.	
\end{proof}

\begin{definition}
	Two symplectic lineal operator $A_i: L^1 \rightarrow L^2$, with $i=1,2$. are called equivalent if there exist a lineal symplectic
	$\tau: L^1 \rightarrow L^2$ such that
	$A_2=\tau A_2 \tau^{-1}$.
\end{definition}
For the purposes of this work we need to verify that the reduced monodromy operator does not depends on the point $m \in T_m(\Lambda^k)$. They are precisely these operators
which will serve to establish sufficient conditions for the existence of the complex germ. Notice $\Xi_j$, with $j=1,\ldots,k$ are symplectic operators then by Definition \ref{redu}, operator $\Xi_m$ in \eqref{redu1} is symplectic as well.

It is valid the following
\begin{proposition}
	Let $m,m^{'} \in \Lambda^k$ then operators  $\Xi_m$ and  $\Xi_{m^{'}}$ defined in \eqref{redu1} are equivalent.
\end{proposition}
\begin{proof}
	There exist $\overrightarrow{t} \in \Re^k$ such that $m=g^{\overrightarrow{t}}m^{'}$. 
	Let the operator $G_{m,m^{'}}=(g^{\overrightarrow{t}})_{*,m}: T_m(M^{2m} \rightarrow T_{m^{'}}(M^{2m})$, which
	is symplectic (see Section \ref{Chapter2}  ). Also $G_{m,m^{'}}(T_m(\Sigma))=T_{m^{'}}(\Sigma)$ and $G_{m,m^{'}}(T_m(\Lambda^k))=T_m(\Lambda^k)$ hold. Then we define the operator $\tau_{m,m^{'}}: \Gamma_m \rightarrow \Gamma_{m^{'}}$ by $\tau_{m,m^{'}}([\theta])=[G_{m,m^{'}}(\theta)]$. We can check that
	$\Xi_{m^{'}}\tau_{m,m^{'}}=\tau_{m,m^{'}}\Xi_m$.
\end{proof}

\begin{remark}
	Analogously it can prove that the operators $\Xi_j$, with $j=1,\ldots,k$ are equivalents at different points
	in the torus $\Lambda^k$.
\end{remark}

Let $\Pi : T_m(\Sigma) \rightarrow \Gamma_m$ the canonical map that each $\theta \in T_m(\Sigma)$ correspond its equivalence class $[\theta]$ modulo $T_m(\Lambda^k)$. 

We have that is valid

\begin{proposition}
	It is valid that
	\begin{equation}
  \Xi_j o \Pi = \Pi o G_j.
  \label{rel1}
	\end{equation}
\end{proposition}
\begin{proof}
	Let $\theta \in T_m(\Sigma)$; $\Pi(G_j(\theta))=[G_j(\theta)]$. Besides that $\Xi_j(\Pi(\theta))=\Xi_j([\theta])=[G_j(\theta)]$, being proved the proposition.
\end{proof}

\begin{definition}
	A complex linear subspace $R$ $\subset$  $\mathbb{C}^n$ is called positive if~~ $\forall x \in R; x \neq 0$ 
	$(1/2i)[x,\bar{x}] > 0$ holds. It subspace is called negative if 	$(1/2i)[x,\bar{x}] < 0$.
\end{definition}

\begin{lemma}
	Let $R \subset \Gamma_m$ a lineal, positive, lagrangian and invariant respect to the operators
	$\Xi_i$, $i=1,\ldots,k$. then the subspace $r^n=\Pi^{-1}(R)$ is dissipative respect to $T_m(\Lambda^k)$, lagrangian and invariant respect of the operator 	$G_i$, $i=1,\ldots,k$, where
	$\Pi$ is the canonical map.
	\label{lemd1}
\end{lemma}
	\begin{proof}
		We have that $\Pi(x)=[x]mod T_m(\Lambda^k)=\{x+y: y \in T_m(\Lambda^k)\}$;\\
	
	a) Let us prove that $r^n=\Pi^{-1}(R)$ is isotropic. Let $x_1,x_2 \in \Pi^{-1}(R)$ then
	$\Pi(x_1), \Pi(x_2) \in R; [\Pi(x_1),\Pi(x_2)]=0$ for  being $r^n$ isotropic; by definition
	we have $[x_1,x_2]=[\Pi(x_1),\Pi(x_2)]=0$, therefore $r^n$ is isotropic.
	
	b) Let us prove that $r^n$ is lagrangian in $T_m(M^{2n}), i.e. dim (R)=(1/2)dim(\Gamma_m)=n-k,\Pi^{-1}(0)=T_m(\Lambda^k)$, then $dim(\Pi^{-1}(0))=k$, also since
	$\Pi^{-1}(0) \subset \Pi^{-1}(R)$, then  and $dim(r^n)=dim(Ker(\Pi))+dim(Im(\Pi^{-1}(R)))=k+n-k=n$, then $r^n$
	is lagrangian.
	
	c) Let us verify that $r^n$ is dissipative respect to $T_m(\Lambda^k)$: we have that $T_m(\Lambda^k) \subset r^n$. Let $x \in r^n \diagdown T_m(\Lambda^k)$ then  $\Pi(x)=[x]\neq 0$. Besides we have $[\bar{x}]=\bar
	{[x]}$, the dissipative condition of $r^n$ is a consequence of the positivity of $R$, i.e.
	\begin{equation*}
	(1/2i)[x,\bar{x}]=(1/2i)[[x],[\bar{x}]]=(1/2i)[[x],\bar{[x]}]=(1/2i)[\Pi(x),\Pi(x)]>0,
	\end{equation*}
	due to $R$ is positive.
	
	d) The invariance of $r^n$ respect of $G_i$ ($i=1,\ldots,k$) is a consequence of \eqref{rel1} and the invariance of $R$ respect to $\Xi_j$ ($i=1,\ldots,k$).
\end{proof}
It is valid the following 
\begin{theorem}
	Let $m \in \Lambda^k$ fixed. If there exist a lineal subspace $N$ lagrangian, dissipative respect to
	$T_m(\Lambda^k)$, invariant respect all the operators $G_i$ ($i=1,\ldots,k$) then the exist a complex germ respect to the Hamiltonian system $x^{'}=JdF_j(x)$ ($j=1,\ldots,k$).
	\label{thep}
\end{theorem}
\begin{proof}
	The idea is to construct a smooth map $r^n : m \rightarrow r^n(m)$ such $r^n(m) \subset T_m(M^{2n})$,
	$\forall m \in \Lambda^k$ satisfying Definition \ref{def1}.
	Let us define $r^n(m)=N$. Since that map $f : \vec{t} \rightarrow g^{\vec{t}}m$, $\vec{t} \in \Re^k$ is subjective,  we have that $\forall p \in \Lambda^k$, there exist $m \in \Lambda^k$ and $\vec{s} \in \Re^k$ such
	that $g^{\vec{s}}m=p$. Let us put 
	\begin{equation}
     r^n(p)=(g^{\vec{s}})_{*,m}(r^n(m)).
     \label{germ1}
	\end{equation}
	Since the operator $(g^{\vec{s}})_{*,m}$ is symplectic and carries the subspace $T_m(\Lambda^k)$ in $T_p(\Lambda^k)$
	then $r^n(p)~~ \forall p \in \Lambda^k$ is lagrangian, disipative with respect to  $T_p(\Lambda^k)$.
	
	Now, we verify that  $r^n(m)$ is invariant: 
	By Definition of Germ in \eqref{germ1} we have
		\begin{equation}
		(g^{\vec{t}})_{*,m}(r^n(m))=(g^{\vec{t}})_{*,p} o (g^{\vec{s}})_{*,m}(r^n(m)).
		\label{germ2}
		\end{equation}
	Using $g^{\vec{t}+\vec{s}}=g^{\vec{t}} o g^{\vec{s}}$ in \eqref{germ2} we obtain
		\begin{equation}
	(g^{\vec{t}})_{*,p} o (g^{\vec{s}})_{*,m}(r^n(m))=(g^{\vec{t}+\vec{s}})_{*,m}(r^n(m)).
	\label{germ4}
		\end{equation}
	Using again \eqref{germ1} in \eqref{germ4}, we obtain the invariance
		\begin{equation}
		(g^{\vec{t}+\vec{s}})_{*,m}(r^n(m))=r^n(g^{\vec{t}+\vec{s}}m)=r^n(g^{\vec{t}} o g^{\vec{s}}m)=r^n(g^{\vec{t}}p).
		\label{germ5}
		\end{equation}
		On the other hand, since $(g^{0})_{*,m}=E_{2n}$, where $E_{2n}$ is identity map, then by choosing
		appropriately the vector $\vec{t}$ we obtain $(g^{\vec{t}})_{*,m}=g_j^tm$, where $t \in \Re$ and
		$g^t_j$ is the Hamiltonian flux associated to the function $F_j$, with $j=1,\ldots,k$. As a consequence
		the Germ is invariant respect to $g^t_j$. The smoothness of the Germ we proof in the Appendix A. 
\end{proof}

\begin{remark}
	From Lemma \ref{lemd1} we have that is possible to construct the complex germ if the
	operators $\Xi_i$, $i=1,\ldots,k$ have a common positive, lagragian and invariant (P.L.I) linear subspace.
	\label{remhj1}
\end{remark}

\begin{remark}
	Proof of Theorem \ref{thep} consists in the construction of complex germ analogously as done
	in \cite{bal2}. In this way, we proof the the map $r^n : m \rightarrow r^n(m)$ is smooth.
\end{remark}

Now following Remark \ref{remhj1}, we find that sufficient conditions on operators $\Xi_i$, $i=1,\ldots,k$  such that they has
a common P. L. I subspace.

\begin{definition}
	Let $(M^{2n},\omega^2)$ a symplectic manifold. The subspace $L \subset T_m(M^{2n})$ in a point $m \in M^{2n}$ is simplectic
	if the restriction of $\omega^2$ a $L$ is no-degenerate.
	\label{delnod1}
\end{definition}
 
\begin{definition}
	A lineal transformation $S L_1 \rightarrow L_2$ between two lineal spaces is called stable if $\forall \epsilon >0$ 
	$\exists \delta>0$ such that $|x| < \delta$ then $|S^n(x)|<\epsilon$ $\forall n  \mathbb{N}, n>0$ (see \cite{arnold2}).
\end{definition}

\begin{proposition}
	A symplectic map is  stable if and only if all its eigenvalues belong to the unitary circle and $S$ is diagonalizable.
	\label{prop1}
\end{proposition}
The proof can be found in \cite{bal2}, where also is proved that the stability condition of symplectic operator is equivalent to have a P. L. I subspace. With this idea, we present previous Lemmas used to prove sufficient condition the existence of a P.L.I subspace for the  monodromy operators $\Xi_i$, $i=1,\ldots,k$.

Let $A$ y $B$ two stable operators.
Let $K_1=\{\sigma_1,\ldots,\sigma_{2n}\}$ the eigenvalues of $A$, where $1$ and $-1$  may also be included.
Therefore, we can do the partition $K_1=K_a \cup \{-1,1\} $, where $K_a=\{\sigma_1,\ldots,\sigma_r,\bar{\sigma_1},\ldots,\bar{\sigma_r}: Im \sigma_i \neq 0\}$ with
$r <=n$ and $\sigma_i$ distinct.
Analogously for the operator $B$ we have the partition  $K_2=K_b \cup \{-1,1\} $, where $K_b=\{\mu_1,\ldots,\mu_s,\bar{\mu_1},\ldots,\bar{\mu_s}: Im \mu_i \neq 0\}$ with
$s <=n$ and $\mu_i$ distinct.

Let us denote by $S_{\sigma}$ the subspace associated to the eigenvalue $\sigma$. In general the eigenvalues
$1$ and $-1$ they may not be included, these will be analyzed separately.

The following Lemma are valid

\begin{lemma}
	Let $A$ y $B$ two stable operators.
	Let us consider the restriction of the operator $A_j=A_{|S_{\sigma_j} \oplus S_{{S_{\bar \sigma_j}}}}$ and $B_j=B_{|S_{\sigma_j} \oplus S_{{S_{\bar \sigma_j}}}}$, where $\sigma_j \in K_a$. Then there exist
	a common P. L. I subspace for $A_j$ and $B_j$. 
		\label{lem1ab}
\end{lemma}
\begin{proof}
	Let us consider the subspace $L=S_{\sigma_j} \oplus S_{{S_{\bar \sigma_j}}}$ which is symplectic and
	invariant respect to the operator $A$ (see \cite{bal2}), i.e. $AL=L$. We have that if $p$ the multiplicity of
	$\sigma_j \in K_a$ then $\dim(L)=2p$.
	
	a)Now we prove that the subspace $L$ is also invariant for the operator $B$. Since $A_jB=BA_j$ we have
	$A_jBL=BL$ therefore $BL$ is invariant for the operator $A_j$ and since $B$ is a symplectic diffeomorphism 
	then $\dim(BL)=2p$. 
	
	Let $S_{\sigma_j}=\{A_{j}x=\sigma_jx\}$, then for $x \in S_{\sigma_j}$ we have $B(A_j(x))=\sigma_jBx$.
	Using $BA_j=A_jB$ we obtain $A_j(Bx)=\sigma_jB(x)$ therefore $Bx \in S_{\sigma_j}$ and as a consequence
	$BL \subset L$. Using that the operator $B$ is a diffeomorphism we have $BL=L$ and the operator
	$B_j=B_{|L}$ is well defined.
	
	b) Now since $B$ is diagonalizable and symplectic operator in $L$, it is possible to obtain a descomposition
	through $K_2$ of L, i.e. 
	\begin{equation}
	L=\sum_{i=1}^{d}(S_{\mu_j}\oplus S_{\bar{\mu_j}})\oplus S_{1} \oplus S_{-1},
	\end{equation}
	where $S_{1}$ and $S_{-1}$ appear if $1$ or $-1$ are eigenvalues of $B$ in $L$. Let us consider
	the restriction ${B_{j}}_{|L_{\mu_j}}$ of the operator $B_j$ to $L_{\mu_j}=S_{\mu_j}\oplus S_{\bar{\mu_j}}$ then the following affirmation are true
		
	1- $L_{\mu_j} ,j=1,\ldots,d$ are symplectic operators and $B_j$ are stable, therefore there exist a subspace P.L.I for
	${B_{j}}_{|L_{\mu_j}}$ in this subspace which we denote by $R_{\mu_j}$ (see \cite{bal2} for the construction of this subpace).
	
	2-The subspace $S_1$ is symplectic  (see \cite{bal2}). Also we have ${B_j}_{|S_1}=Id_{|S_1}$, therefore all vector
	is eigenvalues for ${B_j}_{|S_1}$. Then we can choose a collection of vector in $S_1$ such as they generated a P.L.I. subspace. Let us denote by $R_o$. Also satisfy $AR_o=R_o$. Analogously, for  ${B_j}_{|S_{-1}}=Id_{|S_{-1}}$ is possible to construct a P.L.I subspace denoted by $R_{-1}$.
	
	Finally the subspace 
	\begin{equation}
	R_{\mu}=R_o \oplus R_{-1} \displaystyle \oplus R_{\mu_1}\oplus \ldots \oplus R_{\mu_d} ,
	\end{equation}
	is a common P.L.I subspace to $A_j$ and $B_j$ in $L$. Thus, the proof is a consequence of the way that such subspace are constructed, i.e. it are positive and Lagrangian in $L$. Also, they are invariant for $B_j$,i.e. $B_jR_{\mu}=R_{\mu}$. So, we need to prove that is invariant for $A_j$, i.e. $A_jR_{\mu}=R_{\mu}$ for any $j=1,\ldots,d$.
	
	The subspace $R_{\mu_{j}}$ is constructed  from $h$ eigenvector $e_1,\ldots,e_h$ ($h$ is a multiplicity of $\mu_{j}$
	as eigenvalues of $B_j$ in $L$), which  are associated either $\mu_{j}$ or $\bar{\mu_{j}}$.
	
	Let $e_i$ an arbitrary with $i=1,\ldots,h$. Since $ R_{\mu_1} \subset L$, there $e_i$ is an eigenvalues of $A$ associated
	to either $\sigma_j$ or $\bar{\sigma_j}$. Assume that $A_je_i=\sigma_j e_i$, therefore for $x \in  R_{\mu_1}$,   we have
	$A_jx \in  R_{\mu_1}$.

\end{proof}

\begin{lemma}
	Let $A$ y $B$ two stable operators.
	Let us consider the restriction of the operator $A_1=A_{|S_{\beta}}$ and $B_1=B_{|S_{\beta}}$, where $\beta$ is either $1$ or $-1$ which are eigenvalues of the operator $A$ and $B$. Then there exist a common P. L. I subspace for $A_{1}$ and $B_{1}$. 
	\label{lem2ab}
\end{lemma}
\begin{proof}
	The subspace $S_1$ is symplectic  (see \cite{bal2}). Also we have ${B}_{|S_1}=Id_{|S_1}$, therefore all vector
	is eigenvalues for ${B}_{|S_1}$. Then we can choose a collection of vector in $S_1$ such as they generated a P.L.I. subspace. Let us denote by $R_o$. Also satisfy $AR_o=R_o$. Analogously, for  ${B}_{|S_{-1}}=Id_{|S_{-1}}$ is possible to construct a P.L.I subspace denoted by $R_{-1}$.
\end{proof}

Due to the condition impose on these operators, we have that Lemma \ref{mono1} is a Corollary of 
\begin{lemma}
Let $A$ and $B$ two stable symplectic  operators that commutate in the symplectic space $\mathbb{C}^{2n}$, then they has a common P.L.I subspace.  
\end{lemma}

\begin{proof}
By proposition \ref{prop1} the operators $A$ and $B$ have its eigenvalues in the unitary circle and are diagonalizable.
Then we have $\mathbb{C}^{2n}=\sum_{i=1}^{d}(S_{\sigma_j}\oplus S_{\bar{\sigma_j}})\oplus S_{1} \oplus S_{-1}$, where
$\sigma_j$, $\bar{\sigma_j}$, $1$ and $-1$ are eigenvalues of $A$ (in general $1$ and $-1$ are not necessarily eigenvalues of $A$). Now from  Lemmas \ref{lem1ab} and	\ref{lem2ab} we obtain that there exist a common P.L.I subspace for the $A$ and $B$.
 \end{proof}

\begin{lemma}
	The reduce monodromy operators  $\Xi_i$, $i=1,\ldots,k$ are stable then they have a common P.L.I. subspace.
	\label{mono1}
\end{lemma} 
\begin{proof}
	The construction of the a common subspace for $k>2$ is similar. After we have for two operators we construct for the rest.
\end{proof}
\subsection{Existence of the Germ}
It is valid the following

\begin{theorem}
If the reduce monodromy operators  $\Xi_i$, $i=1,\ldots,k$ are stable then there exist a complex germ invariant respect Hamiltonian system $(M^{2n},\omega^2,F_j)$, with $j=1,\ldots,k$ on the torus $\Lambda^k$.
\label{thef1}
\end{theorem}

The proof is a consequence of Lemma \ref{mono1} and Theorem \ref{thep}.
	
Further we despited a necessary condition for the existence of the Germ.

\begin{theorem}
Assume that there exist a complex germ on the isotropic torus $\Lambda^k$ invariant respect to the Hamiltonian flow
$\dot{x}=IdF_j(x)$, with $j=1,\ldots,k$. Then the reduce monodromy operators  $\Xi_i$, $i=1,\ldots,k$ are stable.
\begin{proof}
	We have $\forall m \in \Lambda^k$ there exist complex subspace lagrangian $r^n(m) \subset T_n(M^{2n})$ such that
	$r^n(m) \supset T_m(\Lambda^k)$ which is dissipative respect of $T_m(\Lambda^k)$, i.e. $\forall x \in r^n(m) \smallsetminus T_m(\Lambda^k)$ we have $[x,\bar{x}]/2i >0$ (condition \eqref{der2a} of Definition
	\ref{def1}). Also, for $t \in \Re$ we have condition \eqref{der1a} of Definition
	\ref{def1}. Therefore, the flow $g^{\vec{t}}=g^{t_1} o \ldots o g^{t_k}: M^{2n} \rightarrow M^{2n}$, where $t=(t_1,\dots,t_n) \Re^k$ satisfy $g^{\vec{t}}_{*,m}(r^n(m))=r^n(g^{\vec{t}}m)$. For $T \in G$ ($G$ discrete subgroup of $\Re^k$ defined in \label{mono1a}) we have 
	\begin{equation*}
g^{T}_{*,m}(r^n(m))=r^n(g^{T}m),
	\end{equation*}
	which mean that $G_j(r^n(m))=r^n(m)$, where $G_j$, $j=1,\ldots,k$ are the monodromy operators, which with period $T_j$, with $j=1,\ldots,k$  generating the subgroup $G$.
	Let us consider the canonical projection map $\Pi: T_m(\Sigma) \rightarrow \Gamma_m=T_m(\Sigma)/T_m(\Lambda^k)$  and let
	$R=\Pi(r^n(m))$.
	
	Now only rest to prove that the subspace $R$ is lagragian, positive in $\Gamma_m$ and invariant respect of the operators
	$\Xi_j$ with $j=1,\ldots,k$.
	Since $\Pi^{-1}(0)=T_m(\Lambda^k)$ we have $\dim(R)=\dim(r^n(m))-k=n-k=(1/2)\dim(\Gamma_m)$. From the definition of symplectic structure in the quotient space $\Gamma_m$ we obtain that $R$ is isotropic and a consequence lagragian en $\Gamma_m$.
	
	We have that if $\Pi(x)\neq 0$ then $x \not\in T_m(\Lambda^k)$. Using that $R$ is disipative respect of $T_m(\Lambda^k)$
	and that $\overline{\Pi(x)}=\Pi(\bar x)$ we have $(1/2i)[x,\bar x]=[\Pi(x),\Pi(\bar x)]=[\Pi(x),\overline{\Pi(x)}]>0$ and as consequence 
	$R$ is positive.
	
	Since $\Xi_j o \Pi=\Pi o G_j$, with $j=1,\ldots,k$ follows that $R$ is invariant respect of $\Xi_j$, with $j=1,\dots,k$.
	Thus we obtain a P.L.I subspace common for the operators  $\Xi_j$, with $j=1,\dots,k$ therefore this operators are
	diagonalizables and its eigenvalues belong to unitary circle then by Proposition \ref{prop1} are stables.
	
\end{proof}

\end{theorem}

\section{Uniqueness of the Germ}

In this section we discuss about sufficient and necessary conditions of complex Germ. This issue is crucial since the germ is used in the construction of an asymptotic quase-classic it necessary to verify if obtained in this way is unique, which depend on the uniqueness of the germ. To display here a full characterization we summarized some basic concept.

\begin{definition}
An stable map $S: L_1 \rightarrow L_2$ is called strong stable if all close map is stable.
\label{def2a}
\end{definition}
What means close map in above Definition?. Let consider the group of the symplectic map which constitute a submanifold of
the lineal map of $\Re^n$. We consider any distance between two lineal map on $\Re^n$ as a distance between the respective 
matrix in a prefixed basis. i.e. Let $[s_{ij}]$ and $[s^{'}_{ij}]$ the matricial representantin of
two lineal map $L_1, L_2 : L^1 \rightarrow L^2$ then they are close if $\max |s_{ij}-s^{'}_{ij}| \le \epsilon$, $\forall \epsilon >0$ and $i,j=1,\ldots,k$. 

Let $A: \mathbb{C}^{2n} \rightarrow \mathbb{C}^{2n}$ an symplectic operator and $\sigma$ an eigenvalues of $A$. We denote by
$L_{\sigma}$ the maximal invariant subspace respect to $A$ associated to $\sigma$.

\begin{definition}
	The eigenvalues $\sigma$ is called elliptic positive (negative) if the subspace $L_{\sigma}$ is positive (negative).
\end{definition}  

In \cite{bal2} a collection of results of the elliptic eigenvalues are obtained. Here we summarize those we will use
\begin{proposition}
	An symplectic map  $L: \mathbb{C}^{2n} \rightarrow \mathbb{C}^{2n}$ is strong stable if all its eigenvalues are elliptic belong to the unitary circle. 
\end{proposition} 
Also the following proposition is valid

\begin{proposition}
	An symplectic map  $L: \mathbb{C}^{2n} \rightarrow \mathbb{C}^{2n}$ is strong stable if has a unique P.L.I subspace.
	\label{uniq1}
\end{proposition}
The following Lemma holds
\begin{lemma}
	If there exist a unique common subspace P.L.I to the reduce monodromy operator $\Xi_j$, with $j=1,\dots,k$ then
	there exist a unique complex germ on the torus $\Lambda^k$. 
	\label{lemf1}
\end{lemma}
\begin{proof}
	Assume that there exist two distinct complex germ $r_1^n$ and $r_2^n$, i.e. there exist a point $m_o$ such that
	$r_1^n(m_o)\neq r_2^n(m_o)$. Let $\Pi$ the canonical subjective and let the subspace $R_j=\Pi(r_j(m_o))$, $j=1,2$ which are by definition
	lagragian and positive in $\Gamma_{m_o}$ and invariant respect the operators  $\Xi_j$, with $j=1,\dots,k$.
	Using analogous process to describe in Lemma \ref{lem1ab} is possible to construct more that one common P.L.I subspace to the operators $\Xi_j$, with $j=1,\dots,k$.
\end{proof}

Now we give sufficient condition for the existence of a unique complex germ.

\begin{theorem}
	In the reduce operator of monodromy are stable and there exist at least one strong stable then there exist a unique complex Germ on the torus $\Lambda^k$ invariant respect to the Hamiltonian system $(M^{2n},\omega^2,F_j)$ with $j=1,\ldots,k$.
\end{theorem}
\begin{proof}
	Let $\Xi_j$ for some $j \in J={1,\ldots,k}$ a strong stable operators. By Proposition \ref{uniq1} there exist a unique
	P.L.I subspace for  $\Xi_j$ which we denote by $R$. As the operator $\Xi_i$, with $i \in I\diagdown{j}$ commute with $\Xi_j$ then by Lemma
	\ref{mono1} is possible to construct a common P.L.I subspace for these operators, which is a unique.
	Then by Lemma \ref{lemf1} we have a Theorem.
\end{proof}

A necessary condition for the existence of germ is given the following
\begin{lemma}
	If there exist a unique complex germ on the torus $\Lambda^k$ invariant respect to the Hamiltonian system $(M^{2n},\omega^2,F_j)$ with $j=1,\ldots,k$. Then there exist a unique common P.L.I subspace for the reduce monodromy 
	operators $\Xi_j$, with $j=1,\dots,k$.
	\label{finlem}
\end{lemma}
\begin{proof}
Let the stable reduce monodromy operators $\Xi_j$, with $j=1,\dots,k$. Then by Lemma \ref{mono1} these operators have for a common P.L.I subspace $R \subset \Gamma_m$ for each $m \in \Lambda^k$. 

 From Theorem \ref{thep} (with this subspace), we can construct the complex germ $r^n=r^n(R)$. Besides, to distinct subspace correspond distinct Germ, i.e.
if $R_1 \neq R_2$ then for $r^n_1=r^n(R_1)$ and  $r^n_2=r^n(R_2)$ we have $ \forall m \in \Lambda^k$ that
$r^n_1(m) \neq r^n_2(m)$.
Assume that there exist $m \in \Lambda^k$ such $r^n_1(m) = r^n_2(m)$; let $m^{'} \in \Lambda^k$ which serves to construct the
reduce operator $\Xi_{m^{'}}$ and the subspace $R_1$ and $R_2$. Since there exist $\vec{s} \in \Re^k$ such that $g^{\vec{s}}(m^{'})=m$, and if we define $r_1(m)=g^{\vec{s}}_{*,m^{'}}(r^n(m^{'}))$; $r^n_1(m^{'})=R_1$ and 
$r^n_2(m)=g^{\vec{s}}_{*,m^{'}}(r^n(m^{'}))$; $r^n(m^{'})=R_2$ we obtain a contradiction $R_1=R_2$.
\end{proof}

We have the following 
\begin{lemma}
	In there is a unique P.L.I subspace common to the reduce monodromy operators  $\Xi_j$, with $j=1,\dots,k$ then these
	operators are stable and at least one is strong stable.
\end{lemma}
\begin{proof}
	We assume that all the operators  $\Xi_j$, with $j=1,\dots,k$ are stable and none is strong stable then from Lemmas \ref{lem1ab}, \ref{lem2ab} it is possible to construct more of one P.L.I subspace common for two operators, i.e. $R_1$ and
	$R_2$. With the procedure of Lemma \ref{mono1} is possible to construct more of one common P.L.I for the
	operators $\Xi_j$, with $j=1,\dots,k$ . 
\end{proof}

We present a necessary condition for the complex germ

\begin{theorem}
	If there is a unique  complex germ invariant respect to the Hamiltonian system $(M^{2n},\omega^2,F_j)$ with $j=1,\ldots,k$.
	Then  the reduce monodromy operators  $\Xi_j$, with $j=1,\dots,k$ are stable and at least one is strong stable.
\end{theorem}

The proof follow from the Lemmas \ref{lemf1} and \ref{finlem}.

\section{Smoothness of the complex Germ}

A p-distribution  on a manifold $M$, with $p \le dim(M)$ is a map such that to each point $m \in M$ correspond a subspace p-dimensional $\theta(m)$ of the $T_m(M)$. We say that such distribution is smooth if there are p smooth vectorial fields 
$X_1,\ldots,X_k$ defined on a neighbor $U$ of the point $m$ such that $X_1(m),\ldots,X_k(m)$ generate the subspace $\theta(m)$.

Since the complex Germ is a n-dimensional distribution with certain additional properties (see Definition ) the smoothness property is a consequence of the own construction of the germ. From remark \ref{remhj1}  the germ is generate by a P.L.I subspace $R$.
We assume that this subspace is generated by the vector $r_1,\ldots,r_n$. Since the map $(g^t_{F_j})_{*,m}$ are smooths the
the vectorial fields $X_j$, $j=1,\ldots,k$ defined on each point $m \in \Lambda^k$ by
\begin{equation}
X_j(m)=(g^{\vec{t}})_{*,m_o}r_j, 
\end{equation}
where $\vec{t} \in \Re^k$ satisfy $g^{\vec{t}}m_o=m_o$, with $g^{\vec{t}}=g^{t_1}_1 o\ldots,g^{t_k}_k$ and
$g^{t_1}_j$ are Hamiltonian flows associated to the Hamilton functions $F_j$.
It is possible to check that $X_j$, $j=1,\dots,k$ are smooth and generated the germ $r^n(m)$.

\section{Application to the Hamiltonian system with cyclic variables}
\label{ultim1}

In this section we study a $2n$-dimensional symplectic manifold that contains a 2k-dimensional manifold ($k<n$) which consist of
$k$ invariant isotropic torus respect to certain Hamiltonian system. These system in the local coordinate  $(I,p,\theta,q)$ has
the form $H=H(I,p,q)$ and the rest are $k-1$ coordinates. These are called system with cyclic variables.
The basis and the problem are described following.  Let $(M^{2n},\omega^2)$ a symplectic manifold with canonical variables ${p,q}.
$.
\begin{definition}
	An atlas on the manifold $(M^{2n},\omega^2)$ is called symplectic if the the coordinate space $(p,q) \in \Re^{2n}$ the symplectic structure
	take the form $\omega^2=\sum_{i=1}^{n}dp_i \wedge dq_i$ and the cart 	$\Phi o \Theta^{-1} : \Theta(U_1\cap U_2) \rightarrow \Phi(W_1 \cap W_2)$ are symplectic transformations, where
	$\Phi : U_1 \rightarrow W_1$, $\Theta : U_2 \rightarrow W_2$,  with $U_1,U_2 \subset M^{2n} $ and $W_1, W_2 \subset \Re^{2n}$,
\end{definition}
\begin{definition}
	A Hamiltonian system $(M^{2n},w^2,H)$ with Hamilton function $H$ is called system with $k$ cyclic variables if there are the system of symplectic coordinates ${I,p,\theta,q}$ such that $H=H(I,p,q)$. The cyclic variables do not appear in the expression
	of the Hamilton function which we denote by $\theta=(\theta_1,\ldots,\theta_k)$.
\end{definition}
Let the variables $I=(I_1,\ldots,I_k)$, $p=(p_1,\ldots,k)$ canonical conjugate of the variables $\theta=(\theta_1,\ldots,\theta_k)mod2\pi$, $q=(q_1,\ldots,q_k)$ with $1<=k<=n$. i.e. the phase space $M^{2n}$ have the coordinate system  $(I,p,\theta,q)$ where the symplectic structure $w^2$ can be written in the form 
\begin{equation}
w^2=\sum_{i=1}^{k}dI_i \wedge \theta_i +\sum_{i=1}^{k}dp_i \wedge  q_i ,
\label{sim1}
\end{equation}
We assume that coordinate system define a diffeormorphism from $M^{2n}$ under the space consist of the direct product
of the torus $T^k$ with open region of $\Re^{2n-k}$. In this system the canonical equation of the Hamilton function take the form
\begin{equation}
I^{'}=0, \quad p^{'}=-H_q, \\
\theta^{'}=H_I, \quad q^{'}=H_p,
\end{equation}
where $H_q=(\partial H/ \partial q_1,\ldots,\partial H/ \partial q_k$ and $H_p=(\partial H/ \partial p_1,\ldots,\partial H/ \partial p_k$.

We consider the subset 
\begin{equation}
N=\{(I,p,\theta,q): H_p=H_q=0\},
\label{nn1}
\end{equation}
and we assume that is connected. Let us denote by $\Lambda^k=\Lambda^k(I_o,p_o,q_o)$ the k-dimensional torus defined by the the equalities $\{(I,p,\theta,q):I=I_o,p=p_o,q=q_o\}$.
It is possible to check that N represent the union of such torus.

We consider the restriction of the vectorial Hamiltonian field to the torus $\Lambda^k$ where a trajectory on such field take the form $g_H^t=(I_o,p_o,H_I t+ \theta,q_o)$ defined for all $t \in (-\infty,+\infty)$ because $\Lambda^k$ is compact. It is possible
to check that any torus $\Lambda \subset N$ is invariant respect to this system. Also, from the symplectic structure in \eqref{sim1} we deduce that $\omega^2 |_{\Lambda^k}=0$, i.e. the torus are isotropic. Therefore $H$ is the union of isotropic, invariant respect to the Hamiltonian system $(M^{2n},\omega^2,H)$.

We have the following Definition
\begin{definition}
	We say that F is first integral of the Hamiltonian system with function $H$ if the Poisson braked satisfy $[F,H]$=0.
\end{definition}
It is possible to check that described above the coordinate function $I_1,\ldots,I_k$ are first integral of the system 
$(M^{2n},w^2,H)$.

We have 
\begin{proposition}
	Assume that for any point in $N$ given in \eqref{nn1} the determinant of $Hess(H)$ is distinct of zero , i.e.
	\begin{equation}
	\mathrm{H_{pp}}\, \mathrm{H_{qq}} - \mathrm{H_{qp}}\, \mathrm{H_{pq}} \neq 0.
		\end{equation}
Then $N$ is sub-manifold symplectic 2k-dimensional that has locally the form of a map $(I,\theta) \rightarrow (p,q)$ that does not depend on the coordinate $\theta$, i.e. $p=P(I), q=q(Q)$.
\end{proposition}
The proof of this result can be found in \cite{bal2}.

From now, we assume that $\Lambda^k$ consist of regular points, i.e. $H_I \neq 0$, as a consequence the functions $H, I2, \ldots, I_k$ are k first integral involution. The problem is under what conditions there is a complex Germ on $\Lambda^k$ with respect to the Hamiltonian system defined by Hamilton functions $H, I2, \ldots, I_k$.

The reduce monodromy operators for this case are the following: $G1=(g^t_H)_{*,m}, G_j=(g^t_{I_j})_{*,m}=E_{2n}$, with $j=2,\ldots,k$ and $E_{2n}$ is identity operator of order $2n$. This mean that the invariance condition if only necessary to check for the operator $G_1$. Applying the result of this paper we obtain the sufficient and necessary for the existence of Germ
different from those obtained in those obtained in \cite{arnold2}.

\section{Conclusion}

In this work we give answer to the question about the existence and uniqueness of a complex germ on a isotropic torus invariant 
respect to the Hamiltonian flows defined by k function $F_1,\ldots,F_k$ that stay in involution in the  phase state $M^{2n}$.

We proof that there exist such germ  if and only if the reduce monodromy operator with period
$T_j$, $j=1,\dots,k$ are stable. This germ is unique if at least one operator is strong stable.

The result obtained here were applied to the case of an Hamiltonian with cyclic variables resulting in new condition for the
existence and uniqueness of complex germ.

\end{document}